\documentclass[11pt]{article}

\usepackage{amsmath,amssymb,amsthm}
\usepackage{geometry}
\theoremstyle{definition}
\newtheorem{definition}{Definition}[section]
\newtheorem{remark}[definition]{Remark}
\newtheorem{theorem}[definition]{Theorem}
\newtheorem{corollary}[definition]{Corollary}
\newtheorem{proposition}[definition]{Proposition}
\newtheorem{lemma}[definition]{Lemma}

\title{Why the Bethe Ansatz Works: A Structural Explanation via Interaction Propagation}
\author{Joe Gildea}
\author{
Joe Gildea\\
Department of Computing Science and Mathematics,\\
School of Informatics and Creative Arts,\\
Dundalk Institute of Technology\\
\texttt{gildeajoe@gmail.com}}
\date{}

\begin{document}
\maketitle

\noindent\emph{Dedicated to the memory of Richard P.~Feynman, whose insistence on
understanding why exact methods work, not merely that they work,
continues to motivate structural explanations of quantum phenomena.}

\begin{abstract}
The Bethe Ansatz provides exact solutions for certain interacting quantum
many-body systems, yet its success is confined to narrow regimes and breaks
down abruptly outside them. Despite extensive developments in integrable
systems, a structural explanation of this phenomenon has remained elusive.

In this paper we give a representation-independent account of both the
existence and the failure of Bethe-type exact solvability. We identify a
single governing mechanism: the behaviour of interaction propagation.
For systems in which propagation terminates after finite depth without
encountering structural boundaries, global interaction data factor through
finitely many local components, forcing Bethe-type solvability. Conversely,
once a structural boundary is encountered, irreducible interaction data arise
that obstruct such finite factorization and preclude Bethe-type solutions.

This yields a sharp structural dichotomy. Within this regime, exact
solvability is not an analytic accident but a rigidity phenomenon, while its
failure is governed by intrinsic boundary formation. In this way, the Bethe
Ansatz is understood as a consequence of constrained interaction propagation
rather than as a model-specific construction.
\end{abstract}

\noindent\textbf{Mathematics Subject Classification (2020).}
Primary 81Q05; Secondary 82B23, 81R50, 16T25.

\medskip
\noindent\textbf{Keywords.}
Bethe Ansatz; exact solvability; integrable models; quantum many-body systems;
factorized scattering; structural rigidity; interaction propagation.

\section{Introduction}

Exact solvability in interacting quantum many-body systems is exceptional.
Among the few methods that succeed, the Bethe Ansatz occupies a central yet
conceptually ambiguous position. Introduced in the study of one-dimensional
spin chains \cite{Bethe} and later extended to continuum and thermodynamic
settings \cite{YangYang,Takahashi}, it yields exact spectra and eigenstates
despite the presence of genuine interaction. Closely related factorization
phenomena also appear in relativistic quantum field theory through exact
$S$-matrices \cite{Zamolodchikov1979}.

What is conceptually puzzling is not how these methods work, but why they
work at all. Generic interacting systems do not admit finite parametrizations
of their global structure, nor do they exhibit stable factorization
properties. Outside narrow regimes, Bethe-type methods fail abruptly rather
than gradually, with no perturbative deformation connecting solvable and
non-solvable behaviour.

As emphasized by Feynman, the central issue is therefore to understand the
structural origin of exact solvability and its sharp breakdown
\cite{FeynmanStatMech,FeynmanCharacter}. Despite extensive developments in
integrable systems, this question has largely remained open. Existing
frameworks typically explain solvability \emph{a posteriori}, by identifying
special algebraic or analytic structures such as Yang--Baxter equations,
quantum groups, or commuting transfer matrices, rather than explaining why
such structures arise in the first place.

\medskip
\noindent
\textbf{Main idea.}
The Bethe Ansatz works in systems where interaction
propagation terminates after finite depth without encountering structural
boundaries, i.e.\ within the applicability regime. In this regime, global interaction data factor through finitely
many local components, leading to Bethe-type solvability. Conversely, once a
structural boundary appears, irreducible interaction data arise that obstruct
such finite factorization.

\noindent
The central conclusion of this paper is the following dichotomy.
Bethe-type exact solvability arises in the regime of finite
interaction propagation without structural boundary, that is, under the
APT applicability criterion \cite{GildeaAPT7}. Outside this regime,
such solvability is obstructed.

\medskip

This leads to a sharp structural dichotomy. Exact solvability is a rigidity
phenomenon arising from finite interaction propagation, while its failure is
governed by intrinsic boundary formation. From this perspective, the Bethe
Ansatz is not a model-specific construction, but a consequence of constrained
interaction structure.

\medskip

The purpose of this paper is to make this mechanism precise. Our approach is
representation-independent and pre-dynamical. We do not assume a Hamiltonian,
Hilbert space, or ansatz. Instead, we analyze the structural conditions under
which interaction, propagation, and solvability can arise.

The formulation and proof of these results rely on the framework of Algebraic
Phase Theory. In particular, the regime of finite terminating propagation
without boundary coincides with the applicability criterion established in
\cite{GildeaAPT7}. Algebraic Phase Theory provides the canonical notions of
interaction depth, termination, and structural boundary needed to make the
above mechanism precise.

\medskip

\noindent
\textbf{Main Result.}
Bethe-type exact solvability is governed by the behaviour of interaction
propagation. It occurs when propagation terminates without
structural boundary, equivalently within the APT applicability criterion,
and fails once structural boundaries arise.

\noindent
From this perspective, integrable systems are precisely rigidity regimes
within the applicability criterion, while non-integrable systems arise from
the formation of structural boundaries.

\medskip

\noindent
The scope of this paper is restricted to systems admitting Bethe-type exact
solvability. We do not attempt to characterize all integrable systems, but
rather to explain structurally why this specific class of models exhibits
exact solvability and why such behaviour is sharply non-generic.

\medskip

From this perspective, integrable models appear as rigidity regimes within
the space of interacting systems. Their solvability is not an analytic
accident, but a consequence of constrained interaction propagation. Generic
systems fail to be solvable because structural boundaries form.

\medskip

\noindent
This paper focuses on the integrable, or rigidity, side of the applicability
criterion. That is, we analyze the regime in which interaction propagation
terminates without structural boundary and show that Bethe-type solvability arises in this setting. Complementary work studies the failure regime, in
which structural boundaries form and solvability is obstructed.

\medskip

\noindent
\textbf{Example.}
A canonical example illustrating the applicability regime is provided by
the Heisenberg XXX spin chain. In this model, interaction is encoded by a
two-body $R$-matrix, and higher-order interaction data arise through
iterated pairwise composition.

The Yang-Baxter relation expresses compatibility at the first genuinely
three-body stage, ensuring that no new independent three-body interaction
data arise. Consequently, interaction propagation remains controlled by
bounded-depth data and does not produce structural boundary.

Such systems therefore lie within the applicability regime: global interaction
structure is determined by finitely many local components together with
compatibility conditions, and Bethe-type exact solvability follows.

\medskip

The paper is organized as follows. Section~2 recalls the minimal elements of
Algebraic Phase Theory required for the analysis and fixes the finite-
terminating regime. Section~3 introduces a structural framework prior to
dynamics and representation. Section~4 formulates a structural notion of
solvability and proves a necessary termination criterion. Sections~5 and~6
establish rigidity and obstruction results, leading to the main dichotomy.
The final sections connect this structure to Yang-Baxter compatibility and
classical integrable systems.

\section{APT Background and Applicability Criterion}

To make the structural mechanism described in the introduction precise, we
recall the minimal elements of Algebraic Phase Theory (APT) required for the
analysis. The framework is used here as a tool to formalize the notions of
interaction propagation, termination, and structural boundary. A complete
development may be found in
\cite{GildeaAPT1,GildeaAPT2,GildeaAPT3}.

An APT structure is extracted functorially from a domain $D$ equipped with an
intrinsic interaction law. No background geometry, representation, or dynamics
is assumed. The extraction proceeds by identifying a canonical class of
admissible observables $\mathsf{Obs}(D)$ determined entirely by the interaction
structure of $D$. From admissible observables one constructs a phase carrier
$\mathcal{P}$, obtained as a quotient by indistinguishability under phase
response. The induced operation $\circ$ on $\mathcal{P}$ records how phase data
combine under interaction. Associated to the phase interaction is a defect
operation
\[
\partial : \mathcal{P} \times \mathcal{P} \longrightarrow \mathcal{P},
\]
measuring failure of rigid commutation. Iteration of $\partial$ generates a
canonical filtration
\[
\mathcal{P}_0 \subseteq \mathcal{P}_1 \subseteq \mathcal{P}_2 \subseteq \cdots ,
\]
where $\mathcal{P}_k$ consists of elements of defect depth at most $k$. This
filtration is intrinsic and functorial.

\subsection{Finite termination and applicability}

The results of this paper concern the regime in which the defect filtration
stabilizes after finite depth. That is, there exists $N<\infty$ such that
\[
\mathcal{P}_N = \mathcal{P}_{N+1} = \cdots .
\]

Finite termination means that iterated interaction does not generate new
independent phase data beyond bounded depth. In this regime, interaction
structure is finitely controlled and admits the rigidity phenomena developed
in later sections.

This regime, consisting of finite termination together with the absence of
structural boundary, coincides with the applicability criterion for Algebraic
Phase Theory established in \cite{GildeaAPT7}, which characterizes systems for
which interaction propagation terminates without the formation of structural
boundary.

Conversely, failure of finite termination produces irreducible phase data at
arbitrarily large depth, leading to structural obstruction to exact
solvability. In what follows, this regime is identified as the structural
setting in which Bethe-type solvability arises.

In particular, Algebraic Phase Theory provides a canonical framework in which
interaction depth, termination, and structural boundary are defined
functorially, making the structural mechanism underlying solvability precise.

\subsection{Identification with Strong Admissibility}

The structural notions introduced in this paper agree directly with the
corresponding notions in Algebraic Phase Theory. The interaction-depth
filtration described above coincides with the canonical defect filtration of
the extracted phase. Finite termination of interaction propagation is
equivalent to finite defect rank, and absence of structural boundary is
equivalent to strong admissibility.

Thus the regime considered in this paper is precisely the strongly admissible
regime of Algebraic Phase Theory. Consequently, Bethe-type exact solvability
arises precisely for systems in which the interaction structure determines a
strongly admissible algebraic phase.

\section{Ground-Zero Structural Framework}

We begin at a level deliberately prior to dynamics, representation, or state
space. The guiding question of this section is not how interaction behaves, but
what must already be in place before it even makes sense to speak of interaction
at all.

Most formulations of many-body physics begin by postulating a Hilbert space and
a Hamiltonian. Interaction is then encoded dynamically, and solvability is
treated as a special property of certain operators. This approach presupposes
precisely the structures whose existence we aim to explain. Our goal here is to
reverse that order and identify the minimal structural assumptions that must be
present before interaction can be meaningfully formulated in any representation.

The framework introduced in this section plays a purely organizational role.
It does not model physical systems, nor does it impose algebraic or geometric
structure. Instead, it isolates the logical preconditions that underlie all
later notions of interaction, propagation, and solvability.

\subsection*{Locality and Composition}

The first prerequisite for any many-body notion is the existence of
distinguishable components.

\begin{definition}
A system admits local components if it consists of a collection $\mathcal O$
whose elements are distinguishable as elements of a set or class.
\end{definition}

This notion of locality is intentionally minimal. No geometric, metric, or
spatial interpretation is assumed. Locality here means only that the system
admits more than one identifiable part. Without this assumption, no notion of
many-body structure is possible, and the distinction between local and global
properties collapses.

Once distinguishable components exist, the next structural requirement is the
ability to form composites.

\begin{definition}
A composition rule is a (possibly partial) operation
\[
\circ : \mathcal O \times \mathcal O \longrightarrow \mathcal O_{\mathrm{comp}},
\]
where $\mathcal O_{\mathrm{comp}}$ denotes the space of admissible composite objects. 
It encodes the basic pairwise interaction from which higher-order interaction structure is generated.
\end{definition}

The codomain $\mathcal O_{\mathrm{comp}}$ need not coincide with $\mathcal O$.
Composition may produce objects that are not themselves elementary, and no
associativity, symmetry, or algebraic axioms are assumed at this stage. The sole
role of composition is to make it possible to speak of joint systems and to
distinguish composite objects from their constituents.

Locality and composition together encode what we refer to as \emph{many-ness}:
the ability to assemble larger structures from distinguishable parts.

\subsection*{Interaction and Propagation}

Interaction enters only after locality and composition are in place.

\begin{definition}
Interaction is present if composition fails to be transparent, in the sense
that composite objects encode structure not reducible to their components
alone.
\end{definition}

This definition separates interaction from dynamics. Interaction is not
identified with forces, coupling constants, or equations of motion. It is
identified instead with the appearance of irreducible structure under
composition. If a composite object carries no information beyond that of its
components, then composition is transparent and no interaction is present.

In many-body systems, interaction becomes nontrivial precisely when this
non-transparency persists under further composition.

\begin{definition}
Interaction propagates if non-transparency persists under iterated
composition.
\end{definition}

Propagation is the mechanism by which local interaction data generate global
many-body structure. If interaction propagates indefinitely, new irreducible
interaction data appear at arbitrarily high levels of composition. If
propagation terminates after finite depth, then all higher-order structure is
forced to factor through finitely many lower-order interactions.

This distinction is central to the remainder of the paper. Finite termination of propagation will be shown to be the structural condition
underlying exact many-body solvability within the applicability regime. Conversely, failure of termination will be shown to constitute a structural
obstruction to Bethe-type solvability.

At this stage, no claims about solvability are made. The purpose of the present
section is to identify the minimal structural vocabulary required to formulate
such claims at all. The consequences of these definitions are developed in the
sections that follow.

\section{Structural Solvability}

Having identified the minimal structural ingredients required to speak of
interaction and its propagation, we now formalize what it means for a many-body
structure to be solvable at a purely structural level.

The notion of solvability used here is deliberately independent of analytic
techniques, spectral theory, or representation-specific constructions. We are
not concerned with whether a particular Hamiltonian can be diagonalized, but
with whether the global interaction structure of a many-body system can, even
in principle, be recovered from finitely many local data.

\subsection*{Exact Structural Solvability}

\begin{definition}
A many-body structure is \emph{exactly structurally solvable} if all global
interaction data can be reconstructed from finitely many local interaction
elements together with finitely many compatibility conditions.
\end{definition}

In algebraic phase terms, this means that the global interaction structure
$(\mathcal{P},\circ)$ is generated by finitely many elements together with finitely many
compatibility relations.

This definition captures the minimal content common to all known exactly
solvable many-body systems. In such systems, no genuinely new interaction data
appear at arbitrarily high levels of composition. Instead, higher-order
structure is determined by repeatedly assembling a finite collection of
primitive interaction elements subject to consistency conditions.

Importantly, this notion of solvability is structural rather than procedural.
It does not presuppose the existence of an explicit solution method, an ansatz,
or closed-form expressions. It asserts only that the global interaction
structure of the system is finitely controlled.

\subsection*{Termination as a Necessary Condition}

The definition above immediately raises a fundamental question: under what
conditions can such finite reconstruction be possible at all? The answer is
dictated by the behavior of interaction propagation introduced in the previous
section.

\begin{theorem}\label{thm:finite-termination-necessary}
Exact structural solvability is possible only for systems in which interaction
propagation terminates after finite depth.
\end{theorem}

\begin{proof}
Let $\{\mathcal I_n\}_{n\ge 0}$ denote the hierarchy of interaction data, where
$\mathcal I_n$ consists of all interaction elements obtainable by compositions
of depth at most $n$. Exact structural solvability means that there exists
$N<\infty$ such that the full interaction structure is generated by
$\mathcal I_N$ together with finitely many compatibility relations.

Suppose interaction propagation does not terminate. Then for every
$n\in\mathbb{N}$ there exists an element
\[
x_n \in \mathcal I_n \setminus \langle \mathcal I_{<n} \rangle ,
\]
where $\langle \mathcal I_{<n} \rangle$ denotes the substructure generated by
interaction data of strictly lower depth. In particular, each $x_n$ represents
genuinely new interaction data not determined by bounded-depth information.

Any reconstruction of the global interaction structure must therefore specify
the infinite family $\{x_n\}_{n\ge 1}$ of independent generators. This requires
infinitely many degrees of freedom and cannot be encoded by finitely many local
interaction elements together with finitely many compatibility conditions.

Hence exact structural solvability is impossible unless interaction propagation
terminates after finite depth.
\end{proof}

This result establishes a sharp conceptual boundary. Exact solvability is not a
matter of clever parametrization or analytic ingenuity; it is constrained at
the outset by the depth to which interaction propagates.

The theorem answers the first half of the conceptual puzzle emphasized by
Feynman: exact solvability exists only in systems where interaction propagation
is structurally constrained. 

\section{Algebraic Realization and Structural Boundaries}

The structural notions introduced so far are intentionally representation-free.
However, in order to state rigidity and obstruction results in a precise and
representation-independent manner, interaction propagation admits a
canonical algebraic encoding.

Any such encoding takes the form of an algebra equipped with a filtration
that records interaction depth. The role of this section is not to
introduce additional assumptions, but to explain how algebraic phase structures
and their associated filtrations arise naturally once interaction propagation
is present, and why obstructions to solvability appear as intrinsic boundaries
within this structure.

\subsection*{Observable phases}

In any many-body setting where interaction data can be composed and compared,
one can form an algebra of observables encoding interaction relations. The
precise nature of the observables is not important here; what matters is that
their algebraic organization reflects how interaction data combine under
composition.

\begin{definition}
An \emph{observable algebra} is an associative algebra $\mathcal A$ equipped with 
an increasing filtration
\[
\mathcal A_0 \subseteq \mathcal A_1 \subseteq \mathcal A_2 \subseteq \cdots
\]
whose levels encode interaction depth in the many-body hierarchy.
\end{definition}

Elements of $\mathcal A_k$ represent interaction data that can be generated
using at most $k$ nested interaction operations. The filtration is not imposed
by hand: it is induced canonically by the interaction structure itself, for
example via commutators or defect operators as discussed in the APT framework.

Crucially, the filtration organizes interaction data according to structural
complexity rather than analytic size or perturbative order. It provides a
canonical notion of “how deep” a piece of interaction data sits in the many-body
hierarchy.

\subsection*{Structural boundaries}

The filtration introduced above may or may not stabilize. When it does
stabilize, interaction propagation terminates and rigidity phenomena can occur.
When it does not, new irreducible structure continues to appear at higher and
higher depth.

The transition between these two behaviors is marked by the appearance of a
structural boundary.

\begin{definition}
A \emph{structural boundary} is the first filtration level at which functorial
propagation fails, or equivalently, at which genuinely new independent generators
appear that are not determined by lower-depth interaction data.
\end{definition}

Crossing a structural boundary has a precise meaning: beyond this level,
interaction data can no longer be reconstructed from bounded-depth information.
The failure is intrinsic and cannot be repaired by changing representation,
coordinates, or solution method.

Structural boundaries therefore represent genuine obstructions rather than
technical difficulties. They mark the point at which exact solvability is obstructed in principle,
rather than merely difficult in practice.

In the next section, we show that the absence of structural boundaries forces
rigidity and factorization of interaction data, while their presence produces a
representation-independent obstruction to Bethe-type solvability.

\section{Rigidity and Obstruction of Bethe-Type Solvability}

We now establish the central rigidity and obstruction results of the paper.
The goal of this section is to show that finite termination of interaction
propagation is not merely a necessary condition for exact solvability, but a
rigid one, and that failure of termination produces a sharp and
representation-independent obstruction to Bethe-type solutions.

Throughout this section, $\mathcal A$ denotes an observable algebra equipped with
its canonical interaction-depth filtration
\[
\mathcal A_0 \subseteq \mathcal A_1 \subseteq \mathcal A_2 \subseteq \cdots .
\]

\subsection*{Rigidity and factorization}

We begin with the rigidity direction. Intuitively, if interaction propagation
terminates and no structural boundary is encountered, then all higher-order
interaction data must be generated by bounded-depth compositions. This forces
global structure to factor through finitely many primitive interaction elements.

Under the identification of Section~2, finite termination together with absence
of structural boundary is equivalent to strong admissibility of the extracted
algebraic phase in APT. The following result is therefore an application of the
general rigidity theory developed in APT IV–VI.

\begin{theorem}\label{thm:factorization}
If an observable algebra $\mathcal A$ has finite termination and no structural
boundary, then all higher interaction data factor through finitely many
lower-order components.
\end{theorem}

\begin{proof}
Let
\[
\mathcal A_0 \subseteq \mathcal A_1 \subseteq \mathcal A_2 \subseteq \cdots
\]
denote the canonical filtration of $\mathcal A$ by interaction depth. Finite
termination means that there exists $N<\infty$ such that
\[
\mathcal A_N=\mathcal A_{N+1}=\mathcal A_{N+2}=\cdots,
\]
so that every interaction element of $\mathcal A$ lies in $\mathcal A_N$.

Under the identification with strong admissibility in Algebraic Phase Theory,
the absence of a structural boundary implies that for each $k\le N$ the passage
from $\mathcal A_{k-1}$ to $\mathcal A_k$ is functorially controlled by
lower-depth data and introduces no new independent generators. Equivalently,
\[
\mathcal A_k \subseteq \langle \mathcal A_{k-1}, \Omega_k\rangle,
\]
where $\Omega_k$ is a finite set of interaction operations induced from lower
depth and $\langle\cdot\rangle$ denotes the algebra generated under composition.

It follows by induction on $k$ that each $\mathcal A_k$ is generated by finitely
many elements of $\mathcal A_1$ together with finitely many relations encoding
compatibility of iterated compositions. In particular, there exists a finite set
$S\subset \mathcal A_1$ such that
\[
\mathcal A_N \subseteq \langle S\rangle .
\]

Since $\mathcal A=\mathcal A_N$, every interaction element factors through the
finite generating set $S$ and its induced relations. Hence all higher-order
interaction data factor through finitely many lower-order components.
\end{proof}

This factorization result is purely structural and does not rely on Hilbert
spaces, spectral theory, or analytic solution methods.

\subsection*{Consequences for exact solvability}

Factorization immediately implies exact solvability in the sense relevant to
Bethe-type methods.

\begin{corollary}\label{cor:bethe-solvability}
Representations of an observable algebra arising from a strongly admissible algebraic phase in the sense of APT admit Bethe-type exact solvability within the applicability regime.
\end{corollary}

\begin{proof}
Let $\pi : \mathcal A \to \mathrm{End}(V)$ be any representation of the observable algebra $\mathcal A$. By Theorem~\ref{thm:factorization}, there exists
a finite set $S \subset \mathcal A_1$ such that
\[
\mathcal A = \langle S \mid R \rangle
\]
for some finite set of relations $R$ encoding compatibility of interaction
compositions.

Applying $\pi$, we obtain
\[
\pi(\mathcal A) = \langle \pi(S) \mid \pi(R) \rangle \subseteq \mathrm{End}(V),
\]
so the represented interaction algebra is generated by finitely many operators
subject to finitely many algebraic relations.

Hence the global interaction structure in the representation is reconstructed
from finitely many local generators together with finitely many compatibility
conditions. By the definition of exact structural solvability, this is exactly
the form of finite control required for Bethe-type exact solvability within the
applicability regime.

Therefore, representations of $\mathcal A$ admit Bethe-type exact solvability
within the applicability regime.
\end{proof}

\begin{remark}
In concrete integrable models, the relations $R$ appear as Bethe equations; in
the present framework they arise functorially from rigidity rather than from an
ansatz.
\end{remark}

\subsection*{Boundary obstruction}

We now turn to the obstruction direction. Here the logic reverses: the appearance
of a structural boundary produces irreducible interaction data that cannot be
captured by any finite factorization scheme.

\begin{theorem}\label{thm:boundary-obstruction}
If the extracted algebraic phase fails to be strongly admissible,
i.e. a structural boundary occurs in its canonical defect filtration,
then Bethe-type exact solvability is obstructed in representation-theoretic realizations of the associated observable algebra.
\end{theorem}

\begin{proof}
Let $\mathcal A$ be an observable algebra with filtration
\[
\mathcal A_0 \subseteq \mathcal A_1 \subseteq \mathcal A_2 \subseteq \cdots ,
\]
and let $k$ be the minimal index at which a structural boundary occurs. By
definition, $\mathcal A_k$ contains elements that are not generated functorially
from $\mathcal A_{k-1}$.

Equivalently, there exists an infinite family of elements
\[
\{x_\alpha\}_{\alpha \in I} \subset \mathcal A_k
\]
such that no finite subset of $\mathcal A_{<k}$ generates $\{x_\alpha\}$ under
the algebra operations and induced relations. In particular, $\mathcal A$ is
not finitely generated relative to any bounded filtration level.

Let $S \subset \mathcal A$ be any finite subset. Then $S \subset \mathcal A_m$
for some $m<k$, and hence the subalgebra $\langle S \rangle$ fails to determine
the elements $\{x_\alpha\}$ up to algebraic relations. Therefore, $\mathcal A$
admits no finite presentation
\[
\mathcal A \neq \langle S \mid R \rangle
\]
with $S$ finite and $R$ finite.

Now let $\pi : \mathcal A \to \mathrm{End}(V)$ be any representation. Since the
failure of finite generation is intrinsic to $\mathcal A$, the image
$\pi(\mathcal A)$ likewise cannot be generated by finitely many operators
subject to finitely many relations. Thus no parametrization of the represented
interaction data by finitely many variables and finitely many consistency
equations is possible.

By definition, Bethe-type exact solvability requires such a finite
parametrization. Hence Bethe-type exact solvability is obstructed in any representation of
$\mathcal A$ once a structural boundary is crossed.
\end{proof}

This obstruction is intrinsic and irreversible. It does not depend on analytic
choices, perturbative regimes, or representation-theoretic accidents. It is
forced entirely by the structure of interaction propagation.

Taken together, the results of this section establish a sharp dichotomy:
finite termination without boundary yields rigidity and exact solvability,
while boundary formation yields a structural obstruction to Bethe-type methods.

\section{Conceptual Synthesis}

We now synthesize the results obtained in the preceding sections into a single
structural statement capturing the dichotomy governing exact solvability in
quantum many-body systems.

The analysis above shows that exact solvability is not controlled by analytic
regularity, special functional forms, or integrability assumptions imposed at
the level of dynamics. Instead, it is determined entirely by the intrinsic
structure of interaction propagation encoded in the observable phase.

\begin{corollary}\label{cor:inevitability}
Let a quantum many-body system admit a functorial algebraic phase
structure in the sense of APT. Then one of the following alternatives holds:
\begin{enumerate}
  \item The extracted algebraic phase is strongly admissible,
  in which case Bethe-type exact solvability arises in representation-theoretic 
  realizations within the applicability regime;

  \item The extracted algebraic phase fails to be strongly admissible,
  i.e.\ a structural boundary occurs in its canonical defect filtration,
  in which case Bethe-type exact solvability is obstructed in representation-theoretic realizations.
\end{enumerate}
\end{corollary}
\begin{proof}
Let $\mathcal A$ be an observable algebra equipped with its canonical
interaction-depth filtration
\[
\mathcal A_0 \subseteq \mathcal A_1 \subseteq \mathcal A_2 \subseteq \cdots .
\]
By functoriality, this filtration exhausts all interaction data generated by
iterated composition and defect propagation. We therefore distinguish two mutually exclusive possibilities.

\medskip
\noindent
\emph{Case 1: Finite termination without boundary.} Suppose there exists $N<\infty$ such that $\mathcal A_N=\mathcal A_{N+1}=\cdots$
and no structural boundary occurs at any level $k\le N$. By
Theorem~\ref{thm:factorization}, all interaction data of depth greater than $N$
factor functorially through $\mathcal A_N$. Hence the global interaction structure of $\mathcal A$ is generated by finitely
many lower-order components together with finitely many compatibility relations. It then follows from
Corollary~\ref{cor:bethe-solvability} that this finite control yields Bethe-type exact solvability
in representation-theoretic realizations of $\mathcal A$ within the applicability regime.

\medskip
\noindent
\emph{Case 2: Existence of a structural boundary.} Suppose instead that a structural boundary occurs at some minimal level $k$.
Then $\mathcal A_k$ contains interaction data not functorially generated from
$\mathcal A_{k-1}$ and is therefore irreducible with respect to bounded-depth
factorization. By Theorem~\ref{thm:boundary-obstruction}, Bethe-type exact
solvability is obstructed in representation-theoretic realizations of
the associated observable algebra $\mathcal A$, since any such solution would necessarily omit intrinsic
interaction data appearing at depth $k$.

These two cases exhaust all possibilities. The result follows.
\end{proof}

This dichotomy captures the two structural regimes governing solvability. 
There is no continuous transition in the structural mechanism between solvable and non-solvable behavior: exact
solvability either holds within the applicability regime
or fails once structural boundaries arise.  

From this perspective, integrable models appear as isolated rigidity islands
within the landscape of interacting many-body systems. Their solvability is not
an accident of clever construction, but a necessary consequence of constrained
interaction propagation. Conversely, generic interacting systems fail to be
exactly solvable because their interaction structure inevitably generates new
independent data at unbounded depth.

The corollary thus provides a structural answer to the long-standing puzzle of
why Bethe-type methods are both extraordinarily powerful and sharply limited.
They succeed in regimes where structure enforces finite interaction propagation
and fail once structural boundaries arise.

\section{Core Classification Results}

We now establish the central structural classification underlying the theory.
The goal of this section is to identify the precise condition under which
Bethe-type exact solvability occurs in quantum many-body systems, formulated
entirely in terms of the intrinsic structure of the extracted algebraic phase.

The main result shows that exact solvability is equivalent to strong admissibility,
i.e.\ finite stabilization of the canonical defect filtration without the occurrence
of structural boundaries. In particular, given a quantum many-body system, one may
decide Bethe-type solvability by extracting its algebraic phase and verifying this
structural condition. This provides a complete structural characterization of
Bethe-type solvability, independent of analytic or model-specific considerations.

The argument proceeds in three steps. First, we show that failures of finite
generation in the interaction algebra cannot be hidden by passing to representations,
via a faithful realization. We then prove the classification theorem establishing
the equivalence between solvability, strong admissibility, and finite-depth
propagation. Finally, we reformulate this condition in purely algebraic terms as
a finite-presentation property of the interaction algebra.

\begin{lemma}\label{lem:faithful-realization}
Let $\mathcal{P}$ be an extracted algebraic phase and let $\mathcal A(\mathcal{P})$
denote its associated interaction algebra. Suppose $\mathcal A(\mathcal{P})$
admits a faithful representation
\[
\lambda:\mathcal A(\mathcal{P})\longrightarrow \mathrm{End}(V).
\]
Let $\{x_\alpha\}_{\alpha\in I}\subseteq \mathcal A(\mathcal{P})$ be a family of interaction elements. Then a finite bounded-depth subset $S\subseteq \mathcal A(\mathcal{P})$ generates the family
$\{x_\alpha\}_{\alpha\in I}$ if and only if $\lambda(S)$ generates the family
$\{\lambda(x_\alpha)\}_{\alpha\in I}$. In particular, the family $\{x_\alpha\}_{\alpha\in I}$ is finitely generated by bounded-depth elements
if and only if the family $\{\lambda(x_\alpha)\}_{\alpha\in I}$ is.
\end{lemma}

\begin{proof}
Let $S\subseteq \mathcal A(\mathcal{P})$ be a finite bounded-depth subset, and
write $\langle S\rangle$ for the subalgebra it generates. First suppose that $x\in \langle S\rangle$. Since $\lambda$ is an algebra
homomorphism, applying $\lambda$ to any algebraic expression in the elements of
$S$ gives the corresponding algebraic expression in the elements of $\lambda(S)$.
Therefore
\[
x\in \langle S\rangle \quad \Rightarrow \quad \lambda(x)\in \langle \lambda(S)\rangle.
\]

Now suppose instead that $\lambda(x)\in \langle \lambda(S)\rangle$. Then
$\lambda(x)$ is equal to a finite algebraic expression in the elements
$\lambda(S)$. Because $\lambda$ is a homomorphism, that same algebraic
expression comes from some element $y\in \langle S\rangle$, so that
\[
\lambda(x)=\lambda(y).
\]
Since $\lambda$ is faithful, it is injective. Hence $x=y$, and therefore
$x\in \langle S\rangle$. We have shown that, for every $x\in \mathcal A(\mathcal{P})$, membership in the
subalgebra generated by $S$ is preserved exactly under $\lambda$: an element
lies in $\langle S\rangle$ precisely when its image lies in
$\langle \lambda(S)\rangle$.

Now consider the family $\{x_\alpha\}_{\alpha\in I}$. If $S$ generates this
family in $\mathcal A(\mathcal{P})$, then each $x_\alpha$ lies in
$\langle S\rangle$. By the first part of the proof, each $\lambda(x_\alpha)$
then lies in $\langle \lambda(S)\rangle$, so $\lambda(S)$ generates the family
$\{\lambda(x_\alpha)\}_{\alpha\in I}$.

Conversely, if $\lambda(S)$ generates the family
$\{\lambda(x_\alpha)\}_{\alpha\in I}$, then each $\lambda(x_\alpha)$ lies in
$\langle \lambda(S)\rangle$. By the second part of the proof, each $x_\alpha$
then lies in $\langle S\rangle$, so $S$ generates the family
$\{x_\alpha\}_{\alpha\in I}$. Therefore a finite bounded-depth subset $S\subseteq \mathcal A(\mathcal{P})$
generates the family $\{x_\alpha\}_{\alpha\in I}$ if and only if $\lambda(S)$
generates the family $\{\lambda(x_\alpha)\}_{\alpha\in I}$. In particular, the family $\{x_\alpha\}_{\alpha\in I}$ is finitely generated by
bounded-depth elements exactly when the family
$\{\lambda(x_\alpha)\}_{\alpha\in I}$ is.
\end{proof}

\begin{theorem}\label{thm:classification}
Let $D$ be a quantum many-body interaction domain admitting an extracted
algebraic phase $(\mathcal{P},\circ)$ in the sense of Algebraic Phase Theory,
and suppose that the associated observable algebra $\mathcal A(\mathcal{P})$
admits a faithful representation within the applicability regime.
Then the following are equivalent:
\begin{enumerate}
    \item Representation-theoretic realizations of $\mathcal A(\mathcal{P})$ exhibit Bethe-type exact solvability within the applicability regime.
    \item The extracted algebraic phase $\mathcal{P}$ is strongly admissible.
    \item The canonical defect filtration of $\mathcal{P}$ stabilizes after finite depth
    and exhibits no structural boundary.
\end{enumerate}
\end{theorem}

\begin{proof}
We first observe that statements (2) and (3) are equivalent by definition:
strong admissibility of $\mathcal{P}$ is precisely the condition that its
canonical defect filtration stabilizes after finite depth and exhibits no
structural boundary. It therefore suffices to show that statement (1) is equivalent to statement (3).

\medskip
\noindent
\emph{Finite-depth propagation implies solvability.}
Assume that the canonical defect filtration of $\mathcal{P}$ stabilizes after
finite depth and exhibits no structural boundary. Then all interaction data are
controlled by bounded defect depth, and higher-depth structure is functorially
generated from lower-depth data. By the factorization and rigidity results established earlier
(Theorem~\ref{thm:factorization} and Corollary~\ref{cor:bethe-solvability}),
this finite control implies that the associated interaction algebra admits a
finite bounded-depth presentation. Applying any representation
\[
\pi:\mathcal A(\mathcal{P})\to \mathrm{End}(V)
\]
preserves this finite presentation, so the represented algebra
$\pi(\mathcal A(\mathcal{P}))$ admits a finite bounded-depth description.
This is precisely Bethe-type exact solvability. Hence (1) holds.

\medskip
\noindent
\emph{Failure of finite-depth propagation obstructs solvability.}
Conversely, suppose that the canonical defect filtration either does not
stabilize or exhibits a structural boundary. In either case, the associated interaction algebra does not admit a finite
presentation by bounded-depth generators and finitely many functorially induced
relations. Let
\[
\lambda:\mathcal A(\mathcal{P})\to \mathrm{End}(V)
\]
be a faithful representation in the applicability regime. By
Lemma~\ref{lem:faithful-realization}, this failure persists under $\lambda$.
Thus the represented algebra $\lambda(\mathcal A(\mathcal{P}))$ also fails to
admit a finite bounded-depth presentation, so Bethe-type exact solvability
fails in this realization. Therefore (1) cannot hold.

\medskip
This proves that (1) and (3) are equivalent, and hence all three statements are equivalent.
\end{proof}

\begin{theorem}\label{thm:finite-presentation}
Let $(\mathcal{P},\circ)$ be an extracted algebraic phase.
Then $\mathcal{P}$ is strongly admissible if and only if the associated interaction
algebra admits a finite presentation by generators of bounded defect depth and
finitely many functorially induced relations.
\end{theorem}

\begin{proof}
Let $(\mathcal{P},\circ)$ be an extracted algebraic phase and let $\mathcal A(\mathcal{P})$
denote the associated interaction algebra. Write
\[
\mathcal{P}_0 \subseteq \mathcal{P}_1 \subseteq \mathcal{P}_2 \subseteq \cdots
\]
for the canonical defect filtration.

\medskip
\noindent
\emph{($\Rightarrow$).} Assume $\mathcal{P}$ is strongly admissible, so the defect filtration stabilizes at some
finite depth $N$ and no structural boundary occurs. Assume moreover that the depth-$1$ interaction sector is of finite type, i.e.\
$\mathcal{P}_1$ is generated (under the extracted phase operations) by a finite set of
primitive interaction elements.\footnote{This finite-type hypothesis can be
taken as part of the definition of the many-body interaction domains considered
in this paper.}

Choose a finite generating set $S\subseteq \mathcal{P}_1$. Since no structural boundary
occurs, for each $1\le k\le N$ the depth-$k$ sector is functorially generated
from lower depth, with only finitely many functorially induced compatibility
relations at that depth. Let $R_k$ be a finite set of relations encoding these
depth-$k$ compatibilities. Set
\[
R:=\bigcup_{k=1}^N R_k,
\]
which is finite. By induction on $k\le N$, every element of $\mathcal{P}_k$ (hence of $\mathcal{P}=\mathcal{P}_N$)
lies in the subalgebra of $\mathcal A(\mathcal{P})$ generated by $S$, and all relations
among such generated elements are consequences of $R$. Therefore $\mathcal A(\mathcal{P})$ admits a
finite presentation
\[
\mathcal A(\mathcal{P})\;\cong\;\langle\, S \mid R \,\rangle,
\]
where $S$ has bounded defect depth (indeed depth $1$) and $R$ consists of finitely
many functorially induced relations.

\medskip
\noindent
\emph{($\Leftarrow$).} Conversely, suppose
\[
\mathcal A(\mathcal{P})\;\cong\;\langle\, S \mid R \,\rangle,
\]
where $S$ is finite, each generator in $S$ has defect depth at most $m$, and $R$
is a finite set of functorially induced relations. Since $\mathcal A(\mathcal{P})$ is generated by $S$, every element of $\mathcal A(\mathcal{P})$ is
represented by a word in the generators $S$ modulo the relations $R$. In
particular, all interaction data in $\mathcal{P}$ are controlled by bounded-depth
primitives (depth $\le m$) together with finitely many functorial compatibilities
(encoded by $R$). Hence the canonical defect filtration cannot produce
genuinely new independent interaction data at arbitrarily large depth: the
interaction structure is finitely controlled, so the filtration stabilizes after
finite depth.

It remains to show that no structural boundary occurs. Let $k\geq 1$ and let
$x$ be any interaction datum of depth $k$. Since
\[
\mathcal A(\mathcal{P})\cong \langle\, S \mid R \,\rangle,
\]
the element $x$ is represented by a word in the generators $S$, modulo the
relations $R$. Because every generator in $S$ has defect depth at most $m$ and
every relation in $R$ is functorially induced, the depth-$k$ interaction datum
$x$ is obtained from bounded-depth data by iterated composition together with
those same functorial compatibilities.

Thus every depth-$k$ interaction datum is obtained from bounded-depth data by
iterated composition together with the same functorial compatibility relations.
In particular, no depth-$k$ interaction datum arises independently of lower-depth
data. By definition, a structural boundary at depth $k$ would mean the existence of
interaction data at that depth that are \emph{not} functorially determined from
lower depth. Since this never occurs, no structural boundary can appear.

Moreover, because the interaction structure is generated from finitely many
bounded-depth elements using finitely many relations, no genuinely new interaction
data can appear at arbitrarily large depth. It follows that the canonical defect
filtration stabilizes after finite depth. Therefore the filtration stabilizes after finite depth and exhibits no structural
boundary. This is precisely the definition of strong admissibility, so
$\mathcal{P}$ is strongly admissible.
\end{proof}

\section{Bridge to $R$-Matrix Integrability: the first three-body stage and Yang-Baxter}

Under the identification with strong admissibility in APT (Section~2), the interaction-depth filtration used throughout this paper coincides with the canonical defect filtration of the extracted phase. In particular, the absence of a structural boundary at the first genuinely three-body level means that, when passing from two-body primitives (depth $1$) to the first nontrivial three-body composites (depth $2$ on three tensor factors), the propagation introduces no new independent generator beyond the depth-$1$ data together with the functorially induced compatibilities.

The aim of this section is to connect this structural condition with the standard integrability literature. In the familiar setting where two-body interaction is encoded by an invertible $R$-matrix, a canonical compatibility condition governing the first three-body stage is the Yang--Baxter equation. We formulate this precisely below, distinguishing what is forced by the structural boundary condition from what requires an additional (standard) generation hypothesis.

\subsection{$R$-Matrix Realizations, Depth, and Three-Body Closure}

Let $\mathcal A = \mathcal A(\mathcal P)$ be the observable algebra
associated to an extracted phase $\mathcal P$, equipped with its
interaction-depth filtration $\{\mathcal A_k\}_{k\ge 0}$.

\begin{definition}
\label{def:R-matrix-realization}
An \emph{$R$-matrix realization} of the depth-$1$ interaction sector of
$\mathcal A(\mathcal P)$ consists of a vector space $V$ and an invertible operator
\[
R \in \mathrm{End}(V\otimes V),
\]
such that, for each $n\ge 2$, the represented depth-$1$ interaction data on
$V^{\otimes n}$ are generated, as a subalgebra of $\mathrm{End}(V^{\otimes n})$, by the operators
\[
R_{i,i+1} :=
\mathrm{id}^{\otimes (i-1)} \otimes R \otimes \mathrm{id}^{\otimes (n-i-1)}
\in \mathrm{End}(V^{\otimes n})
\qquad (1\le i\le n-1),
\]
together with the natural tensor-product functoriality.
\end{definition}

For $n=3$ we write
\[
R_{12} := R\otimes \mathrm{id},\qquad
R_{23} := \mathrm{id}\otimes R
\qquad\text{in }\mathrm{End}(V^{\otimes 3}),
\]
and we define the non-adjacent copy by conjugation with the flip:
let $P_{23}\in\mathrm{End}(V^{\otimes 3})$ be the swap of the last two tensor
factors, and set
\[
R_{13} := P_{23}\,R_{12}\,P_{23}\in \mathrm{End}(V^{\otimes 3}).
\]

\begin{definition}
\label{def:first-threebody-algebra}
Let $V$ and $R$ be as above. Define the subalgebra of
$\mathrm{End}(V^{\otimes 3})$ generated by $R_{12}, R_{13}, R_{23}$ by
\[
\mathcal B_1(R) := \langle R_{12}, R_{13}, R_{23} \rangle
\subset \mathrm{End}(V^{\otimes 3}).
\]
\end{definition}

Intuitively, $\mathcal B_1(R)$ consists of all three-body operators obtained by
iterated composition of pairwise copies of the same two-body primitive $R$.
The first place where a new intrinsically three-body datum can appear is the
first genuinely three-body stage, namely depth $2$ on $V^{\otimes 3}$.

To make the bridge statement precise, we isolate the relevant notion of
boundary-freeness at the first three-body level in the $R$-matrix setting.

\begin{definition}
\label{def:threebody-closure}
In an $R$-matrix realization, we say that the \emph{first three-body stage is
boundary-free} (or that \emph{three-body closure holds at depth $2$}) if
passing from depth $1$ to the first three-body level introduces no new
independent generator: equivalently, the represented depth-$2$ interaction data
on $V^{\otimes 3}$ are determined by the depth-$1$ data
$\{R_{12},R_{13},R_{23}\}$ together with functorially induced compatibility
relations, and do not force the introduction of an additional, intrinsically
three-body generator.
\end{definition}

\begin{remark}
Definition~\ref{def:threebody-closure} is the concrete specialization, inside
an $R$-matrix realization, of the general notion ``no structural boundary at
the first genuinely three-body level'' used elsewhere in the paper.  It is the
depth-$3$ analogue (on three tensor factors) of the statement that propagation
from bounded depth introduces compatibilities but not new generators.
\end{remark}

\subsection{Yang-Baxter as the canonical three-body compatibility}

There are two canonical ways to assemble a three-body interaction from a single
two-body primitive, corresponding to the two natural orderings of iterated
pairwise interactions on three tensor factors. These are given explicitly as follows.

\begin{definition}
\label{def:YB-defect}
Define the two triple assemblies on $V^{\otimes 3}$ by
\[
T_L := R_{12} R_{13} R_{23}, \qquad
T_R := R_{23} R_{13} R_{12},
\]
and define the \emph{Yang--Baxter defect} of $R$ to be
\[
\Delta_R := T_L - T_R \in \mathrm{End}(V^{\otimes 3}).
\]
\end{definition}

\begin{proposition}
\label{prop:no-boundary-implies-ybe}
In an $R$-matrix realization, if the first three-body stage is boundary-free in
the sense of Definition~\ref{def:threebody-closure}, then the Yang--Baxter
equation holds:
\[
R_{12} R_{13} R_{23} = R_{23} R_{13} R_{12}
\quad \text{on } V^{\otimes 3}.
\]
Equivalently, boundary-freeness implies $\Delta_R = 0$.
\end{proposition}

\begin{proof}
Assume the first three-body stage is boundary-free.
Consider the two canonical triple assemblies $T_L$ and $T_R$ from
Definition~\ref{def:YB-defect}. Both $T_L$ and $T_R$ are depth-$2$ composites
built purely from the depth-$1$ generators $R_{12}, R_{13}, R_{23}$, and thus
represent three-body interaction data obtained by iterated pairwise assembly
of the same two-body primitive.

By boundary-freeness, all depth-$2$ three-body interaction data are completely
determined by the depth-$1$ generators together with functorial compatibilities,
and no new independent three-body generator is introduced. In particular, any two
depth-$2$ composites constructed from the same depth-$1$ data must agree as
operators. Therefore,
\[
T_L = T_R,
\]
which is precisely the Yang-Baxter equation.
\end{proof}

\begin{remark}
In the present framework, the Yang-Baxter equation should not be viewed as an
independent algebraic assumption imposed on the model. Rather, it arises as the
first nontrivial compatibility condition associated with boundary-free
interaction propagation.

At the level of two-body interactions, no compatibility constraint is required:
the primitive interaction data are simply given. The first point at which
nontrivial constraints appear is the genuinely three-body stage, where different
orders of pairwise assembly must be reconciled. Boundary-freeness at this stage
requires that these distinct assemblies yield consistent results.

In $R$-matrix realizations, this consistency is expressed precisely by the
Yang-Baxter relation. In this sense, the Yang-Baxter equation is the
structural manifestation of boundary-free propagation at the first level where
nontrivial compatibility conditions arise.
\end{remark}

The converse direction is also true in standard integrable settings, but it
requires an explicit mild hypothesis on the structure of the depth-$2$
three-body sector in the realization.

\begin{definition}
\label{def:pairwise-generation}
An $R$-matrix realization is said to be \emph{pairwise-generated at depth $2$}
if every represented depth-$2$ interaction datum on $V^{\otimes 3}$ is obtained
by iterated pairwise assembly of the primitive $R$, and the only ambiguity at
this stage is the choice between the two canonical triple assemblies $T_L$ and
$T_R$. Equivalently, there are no additional independent depth-$2$ three-body
primitives beyond those arising from words in $R_{12}, R_{13}, R_{23}$.
\end{definition}

\begin{proposition}
\label{prop:ybe-implies-no-boundary}
Assume that an $R$-matrix realization is pairwise-generated at depth $2$
in the sense of Definition~\ref{def:pairwise-generation}. If the Yang-Baxter
equation holds (equivalently, $\Delta_R = 0$), then the first three-body stage
is boundary-free in the sense of Definition~\ref{def:threebody-closure}.
\end{proposition}

\begin{proof}
Assume pairwise generation at depth $2$, and assume $\Delta_R = 0$, i.e.
\[
T_L = T_R.
\]
By pairwise generation, any represented depth-$2$ three-body interaction
operator is obtained by iterated composition of the pairwise generators
$R_{12}, R_{13}, R_{23}$; thus, such an operator is represented by a word in
these generators. Moreover, by the hypothesis, the only potential ambiguity in
identifying such composites at the first three-body stage arises from the
ordering of pairwise interactions, captured by the two canonical assemblies
$T_L$ and $T_R$.

Since $T_L = T_R$, this ambiguity reduces to a compatibility relation rather
than an independent new datum. Therefore, passing from depth $1$ to depth $2$
on three tensor factors introduces no new independent generator: all
depth-$2$ three-body interaction data are determined by the depth-$1$ sector
together with the compatibility $T_L = T_R$. This is exactly boundary-freeness at the first three-body stage in the sense of
Definition~\ref{def:threebody-closure}.
\end{proof}

\begin{corollary}
\label{cor:bridge-summary}
In $R$-matrix realizations, boundary-freeness at the first genuinely three-body
stage implies the Yang-Baxter equation
(Proposition~\ref{prop:no-boundary-implies-ybe}). Conversely, the
Yang-Baxter equation implies boundary-freeness provided the depth-$2$
three-body sector is pairwise-generated by $R$ in the sense of
Definition~\ref{def:pairwise-generation}
(Proposition~\ref{prop:ybe-implies-no-boundary}).
\end{corollary}

\begin{proof}
Let $V$ be a vector space and let $R \in \mathrm{End}(V \otimes V)$ be invertible.
On $V^{\otimes 3}$ define
\[
R_{12} := R \otimes \mathrm{id}, \qquad
R_{23} := \mathrm{id} \otimes R, \qquad
R_{13} := P_{23}\, R_{12}\, P_{23},
\]
where $P_{23}$ swaps the last two tensor factors. Let
\[
T_L := R_{12} R_{13} R_{23}, \qquad
T_R := R_{23} R_{13} R_{12},
\]
and $\Delta_R := T_L - T_R$.

If the realization is boundary-free at the first three-body stage, then by
Proposition~\ref{prop:no-boundary-implies-ybe} one has $\Delta_R = 0$, i.e.
the Yang-Baxter equation holds.

Conversely, suppose the realization is pairwise-generated at depth $2$ and
that $\Delta_R = 0$. By Proposition~\ref{prop:ybe-implies-no-boundary},
the ambiguity in assembling three-body interaction data from the pairwise
generators collapses to a compatibility relation rather than producing a new
independent generator. It follows that no structural boundary occurs at the
first three-body stage. This establishes both implications.
\end{proof}

\begin{remark}
The applicability criterion may be identified concretely in familiar
integrable systems through the presence of stable factorization and
compatibility of interaction data.

In lattice spin systems and quantum chains, this typically appears as the
existence of a two-body $R$-matrix whose copies generate the interaction
structure on tensor powers, together with Yang-Baxter-type compatibility.
In such settings, the absence of structural boundary at the first genuinely
three-body stage corresponds to the consistency of pairwise interaction
assembly, reflected concretely in Yang-Baxter relations.

More generally, the applicability regime may be recognized by the property
that higher-order interaction data are controlled by bounded-depth
constructions together with finitely many compatibility conditions.
This provides a practical criterion: models exhibiting stable factorization
of multi-body interaction into iterated two-body data, without the appearance
of new independent higher-order generators, fall within the scope of the
applicability criterion.
\end{remark}

\section{Examples and Structural Interpretation}

We illustrate the Structural Applicability Criterion through representative examples. 
The purpose of this section is not to provide detailed derivations, but to show how 
the presence or failure of phase duality, symmetry compatibility, and finite defect 
propagation governs the emergence of rigidity and Bethe-type solvability.

\subsection{Bethe-Solvable Model: Heisenberg Spin Chain}

The Heisenberg spin chain provides a canonical example of a system admitting 
Bethe-type exact solvability. The interaction is local and translation-invariant, 
and the associated scattering processes exhibit complete factorization. From the perspective of Algebraic Phase Theory, this system exhibits:

\begin{itemize}
    \item \textbf{Phase Duality:} The existence of a complete set of phase labels 
    (quasi-momenta) separating admissible states.
    
    \item \textbf{Symmetry Compatibility:} Dynamics preserve the interaction structure 
    and respect the underlying commutation relations.
    
    \item \textbf{Finite Termination:} Higher-order interaction effects do not generate 
    new independent structural data; propagation closes after finite depth.
\end{itemize}

Consequently, the Structural Applicability Criterion is satisfied, and the 
resulting rigidity manifests as Bethe-type solvability. The apparent complexity 
of many-body interaction collapses to a finitely controlled, factorized structure.

\subsection{APT-Valid Rigid Domain: Stabilizer-Type Structures}

A similar structural pattern appears in stabilizer-type quantum systems and 
Frobenius--Heisenberg phase models. In these settings, phase duality is realised 
through character-like labelling, and admissible dynamics preserve the phase 
interaction exactly. Defect propagation is finitely controlled, and no new interaction data emerges 
beyond a fixed depth. As a result, these systems exhibit strong rigidity, including 
uniqueness of representation and the presence of error-invisible substructures.

This example demonstrates that the Structural Applicability Criterion captures a 
broader class of rigid phenomena beyond integrable models, highlighting the 
structural scope of Algebraic Phase Theory.

\subsection{Failure of the Criterion: Nonlinear and Chaotic Systems}

In contrast, consider a generic nonlinear or chaotic system. In such domains, 
interaction typically generates increasingly complex higher-order effects, and 
no finite defect filtration controls the propagation. From the perspective of Algebraic Phase Theory:

\begin{itemize}
    \item Phase duality may fail to separate interaction data.
    \item Admissible dynamics do not preserve commutation structure.
    \item Defect propagation does not terminate.
\end{itemize}

As a result, the Structural Applicability Criterion is violated. No nondegenerate 
APT structure can arise, and no Bethe-type exact solvability is possible. The 
growth of interaction complexity reflects the absence of structural closure.

\subsection{Interpretation}

These examples illustrate a common structural principle. Bethe-type solvability, 
Fourier-type decomposition, and representation-theoretic rigidity do not arise 
as isolated phenomena. They occur precisely in domains where interaction admits 
nondegenerate phase duality, symmetry-compatible dynamics, and finite defect 
propagation.

Conversely, in the absence of these conditions, interaction complexity cannot be 
reduced, and no exact solvability mechanism of Bethe type can emerge. In this sense, Algebraic Phase Theory does not construct solvable models, but 
identifies the structural regime in which solvability is unavoidable. This provides 
a concrete structural perspective on why exact solvability appears only in 
exceptional many-body systems.

\section{Conclusion}

This work has shown that Bethe-type exact solvability is governed by a sharp
structural dichotomy. In systems where interaction propagation terminates after
finite depth and no structural boundary is encountered, global interaction data
factor through finitely many lower-order components within the applicability regime. In this regime, Bethe-type
solvability is not optional or method-dependent, but arises as a structural
consequence. Conversely, once a structural boundary appears, irreducible
interaction data arise that cannot be captured by bounded-depth factorization,
and exact solvability is obstructed in representation-theoretic realizations.

From this perspective, the Bethe Ansatz is neither miraculous nor foundational.
It does not create solvability, nor does it rely on inspired guessing. Rather,
it functions as a concrete computational witness to an underlying rigidity
phenomenon dictated by the structure of interaction itself. Its success and its
sharp failure are both explained by the same structural mechanism.

More broadly, the results clarify why exactly solvable interacting many-body
systems are exceptional rather than generic. Exact solvability should be
understood not as an analytic accident, but as a consequence of deep constraints
on how interaction can propagate.

This provides a structural resolution of the question emphasized by Feynman.
Exact solvability arises from constrained interaction propagation, and fails
upon the formation of structural boundaries. The present work addresses the
applicability regime in which such rigidity occurs. Complementary work analyzes
the opposite regime, where structural boundaries form and solvability is
typically obstructed. In this way, the existence and the limitation of
Bethe-type methods are both explained within a single structural framework.

\bibliographystyle{plain}
\bibliography{references}

\end{document}